 \documentclass[11pt,leqno,twoside,a4paper]{amsart}
 \usepackage{amsmath, amsthm, amscd, amsfonts, amssymb, graphicx, color}
\usepackage{a4wide}

 \newtheorem{theorem}{Theorem}[section]
 \newtheorem{lem}[theorem]{Lemma}
 \newtheorem{prop}[theorem]{Proposition}
 \newtheorem{cor}[theorem]{Corollary}
 \theoremstyle{definition}
 \newtheorem{definition}[theorem]{Definition}
 \newtheorem{example}[theorem]{Example}

 \theoremstyle{remark}
 \newtheorem{remark}[theorem]{Remark}
 \numberwithin{equation}{section}

 \newcommand{\N}{\mbox{$\mathbb{N}$}}

 \newcommand{\Z}{\mbox{$\mathbb{Z}$}}

\newcommand{\mS}{\mathcal{S}_R}
\newcommand{\mN}{\mathcal{P}_R}
\newcommand{\mV}{\mathcal{V}}
\newcommand{\ud}{\mbox{\rm udim}}
\newcommand{\ess}{\leqslant_e}
\newcommand{\ann}{\mbox{\rm ann}}
\begin{document}

 \title[    Sum-Essential   Graphs of Modules ]{Sum-Essential   Graphs of Modules}

 \author[J.Matczuk]{Jerzy Matczuk}
  \address{Institute of Mathematics, Warsaw University,  Banacha 2, 02-097 Warsaw, Poland}
\email{\textcolor[rgb]{0.00,0.00,0.84}{jmatczuk@mimuw.edu.pl}}
\author[A.Majidinya]{Ali  Majidinya}
 \address{Department of
Mathematical Sciences, Salman Farsi University of Kazerun, Kazerun,
Iran, P.O. Box 73175-457.}
 \email{\textcolor[rgb]{0.00,0.00,0.84}{ali.majidinya@gmail.com and
ali.majidinya@kazerunsfu.ac.ir}}




 \begin{abstract}
 The  sum-essential   graph  $  \mS(M) $     of
 a left $R$-module $M$ is a graph whose
 vertices are all nontrivial submodules of $M$ and two distinct
submodules are adjacent iff their  sum   is an
essential submodule of $M$. Properties of the graph $\mS(M)$ and
  its subgraph $\mN(M)$ induced by vertices which are not essential
  as submodules of $M$ are investigated.   The interplay between module properties
  of $M$  and  properties of those graphs   is   studied.
\end{abstract}
 \maketitle
 \section*{Introduction} Throughout the paper a ring means a
  unital associative  ring    and a module is a  left unitary module over a given ring $R$.

There are many studies on  various graphs associated to modules,
rings and other algebraic structures.
  The aim of the paper is to investigate
  the interplay between module properties of a module $M$
   and  properties of its sum-essential  and proper sum-essential graphs. This will result in classification of    modules in terms of some specific properties of those  graphs.

  For a given module $M$ the \it sum-essential graph \rm of $M$ is a simple graph $\mS(M)$ (i.e. unweighted, undirected graph containing no graph loops or multiple edges) whose vertices are nontrivial submodules (i.e. different from 0 and $M$)   and two distinct vertices  are adjacent if and only if the sum of the corresponding submodules is an essential submodule of $M$.
 We will also study the  \it proper sum-essential graph \rm of $M$ which is a subgraph of $\mS(M)$ generated by  vertices which, as submodules of $M$ are not essential.

The graph $\mS(M)$ was introduced and studied by  Amjadi in \cite{amja1} in a very
special case when $R$ is a commutative ring and $M={_R}R$.
There are other graphs  relating submodule structure of a given module. For example
  comaximal left ideal graphs,  i.e. graphs  whose vertices are nontrivial left   ideals of a ring $R$ and two vertices $I,J$ are
adjacent if and only if $I+J=R$ were considered in \cite{ami,wa,wan2}.
   The intersection graph of a module, i.e. a simple graph  whose vertices are
     nontrivial submodules and two distinct vertices are adjacent if   the intersection of
     corresponding submodules is nonzero, was introduced and investigated in \cite{red}. Its
        complement graph was  also examined. Earlier those graphs were considered in the context
        of left ideals of a given ring (see \cite{ A2, A1}). In \cite{ak2} the authors considered the inclusion graph whose vertices  were nontrivial left   ideals of a ring.

The first, introductional section contains elementary observations needed later on.
   We also show that    $\mS(M)$ and $\mN(M)$ are connected graphs of diameter not bigger than 3.

 Section 2 concentrates on modules $M$ such that the associated  graphs  contains vertices of small degree.   Theorems \ref{thm finite deg S(M)} and \ref{thm finite deg P(M)} show  that if every vertex  of $\mS(M)$ ($\mN(M)$) is of finite degree  then the graph has only finitely many vertices. An information about the structure of such modules is also given. The remaining part of this section is mainly devoted to  investigation of  modules $M$ such that $\mN(M)$ contains a vertex of degree 1.   Theorem \ref{Theorem degree 1 gen. case} offers necessary and sufficient condition for a submodule $U$ of $M$ to be of  degree 1 as a vertex in $\mN(M)$.  It also appears (see Theorem \ref{Thm largest submod deg 1}) that either all submodules of degree one in $\mN(M)$ are simple or $M$ contains the unique largest submodule of degree 1, examples are presented.

 Section 3  concentrates on modules whose graphs are complete, $k$-regular and triangle-free or a tree.  For example, Theorem \ref{thm complete graph} describes modules $M$  with complete graph $\mS(M)$.   Theorem \ref{triangle1} and \ref{treeta} give  necessary and sufficient conditions for $\mN(M)$ to be triangle-free and a tree, respectively. The girth of $\mS(M)$ and $\mN(M)$ is also described.

\section{\bf Preliminaries}
If not specified otherwise $M$  denotes a left module over a given
ring $R$. We write $A \ess  M$ if $A$ is an essential submodule of
$M$ (i.e. if $A\cap B\ne 0$, for any nonzero submodule $B$ of $M$), $\ud(M)$ stands for the   \it uniform dimension \rm of $M$. For
$m\in M$,   $\ann(m)$ denotes the annihilator of $m$ in $R$.   The reader
can consult \cite{la1} and \cite{la2} for precise definitions and
properties needed in the text.

 Let $G$ be a simple graph. The vertex set of $G$ is   denoted by $\mV(G)$, $\deg_G(v)$ stands for the  \it degree \rm of $v\in \mV(G)$, i.e.
    the cardinality of the set of all vertices  which are adjacent to $v$.   The \it
     maximum and minimum  degrees  \rm   of the graph   $G$  are
      the maximum and minimum degree of its vertices and are
       denoted by $\Delta(G)$ and $\delta(G)$, respectively. $G$ is a complete graph
       if every pair of distinct vertices of $G$ are adjacent, $\mathcal{K}_n$ will stand for a
       complete graph with $n\in \N$ vertices.  The graph $G$ is  \it  $k$-regular\rm , if $\deg_{G}(v)=k<\infty$ for every $v\in \mV(G)$.

  Let  $u, v\in \mV(G)$. We say that $u$ is a \it universal vertex \rm of $G$ if $u$ is adjacent to  all other vertices of $G$ and
     write $u\backsim v$ if $u$ and $v$ are  adjacent. The distance $d(u,v) $ is
      the length of the shortest path from $u$ to $v$ if such path exists, otherwise
   $d(x,y)=\infty$. The \it diameter \rm of $G$   is $\mbox{\rm diam}(G)=sup \{ d(u,v)\mid u,v\in \mV(G)\}$.

  The \it girth \rm of a graph $G$,  denoted by $g(G)$, is the length of a shortest cycle in $G$. If $G$ has no cycles, then   $g(G)=\infty$.

 A subset $S\subseteq \mV(G)$ is \emph{independent} if no two vertices of $S$ are adjacent.
 For a positive integer $k$,
a \it $k$-partite \rm  graph is a graph whose vertices   can be partitioned into $k$ nonempty independent sets.

 Let us recall that if $M$ is a left module over the ring $R$,  then:

\begin{itemize}
  \item  The \it sum-essential graph \rm of $M$ is a simple
   graph $\mS(M)$ such that $\mV(\mS(M))$ consists of all
    nontrivial submodules of $M$ and two vertices  are adjacent if
    and only if the sum of the corresponding submodules is an essential submodule of $M$.
  \item  The \it proper sum-essential graph \rm of $M$, denoted by $\mN(M)$, is a subgraph of
   $\mS(M)$ induced by  vertices which are not essential  submodules of $M$.
\end{itemize}

Let us remark that $\mS(M)$ has exactly one vertex iff $M$ contains a
 submodule $B$ such that both $B$ and $M/B$ are simple modules. On the other
 hand, if the graph $\mN(M)$ is not empty, then the module $M$ is not uniform and
  hence  $|\mV(\mN(M))|\geq 2$.

Henceforth we will assume that all considered modules $M$ have at least two proper
 submodules, i.e. $|\mV(\mS(M))|\geq 2$. Notice that if $N$ is a nontrivial nonessential submodule
 of $M$  then $N$, as a vertex of     $\mS (M)$  is adjacent to
  its complement  and every vertex which is essential as a submodule of $M$
   is a universal vertex of $\mS (M)$. Hence $\delta(\mS (M))\geq \delta(\mN(M))\geq 1$.
   Moreover, if $M$ contains a proper essential submodule  (i.e. $M$ is not a semisimple module),
    then  $\Delta(\mS (M))=|\mV(\mS(M))|-1$. The above remarks will be used freely in the text.
    Using the above it is easy to see that:
\begin{prop} Suppose the module  $M$ is not simple. Then  $M$ is semisimple if and only if
 $\mS(M)=\mN(M)$   if and only if there exists a vertex $X$ of $\mN (M)$, such that
  $\deg_{\mS (M)}(X)=\deg_{\mN (M)}(X)$.

\end{prop}

\begin{example}\label{ex1}  Let $M=\bigoplus_{i\in T}  B_i$ be a
direct sum of pairwise non isomorphic simple $R$-modules, where $T$ is a
finite set with $|T|=n\geq 2$.  Then submodules of $M$ are in one to one
correspondence with subsets of $T$.  Let $A=\bigoplus_{i\in S}B_i$ and
 $B=\bigoplus_{i\in S'}B_i$   be a proper submodules $M$ such that $A+B=M$.
 Then $S'=(T\setminus S)\cup F$, where $F$ is a subset of $S$ different from $S$.
 This shows that  $\deg_{\mS(M)}(A)=2^{|S|}-1 $. In particular, maximal submodules correspond to
   vertices of $\mS(M)$ of degree   $ 2^{n-1}-1=\Delta(\mS(M)) $ and minimal submodules to vertices
   of degree $1=\delta(\mS(M))$.
\end{example}

Notice that if $X, A\in \mV(\mN (M))$  and  $X \subseteq A$, then $\deg_{\mN(M)}(X) \leq \deg_{\mN(M)}(A)$. The following example shows that this inequality can be sharp even when $A$ is a uniform submodule of $M$.
\begin{example} Consider the $\Z$-module
  $M=\mathbb{Z}_{8}\oplus\mathbb{Z}_2$ and set $A=\Z(1,0)=\Z_8$ and $X=4\Z_8\subseteq A$.
   Then $A$ is uniform and  one can check that $\deg_{\mathcal{P}_\mathbb{Z}(M)}(X)=2$ and $   \deg_{\mathcal{P}_\mathbb{Z}(M)}(A)=3$.
     \end{example}
     \begin{example}\label{basic ex}
 Let $M=B_1\oplus B_2$, where $B_i$'s are  simple $R$-modules. Then:\\
 \indent(i) $|\mV(\mS(M)|=|Hom_{R}(B_1,B_2)|+1$ \\
 \indent(ii) $\mS(M)=\mN(M)$ is a complete graph of cardinality $|Hom_{R}(B_1,B_2)|+1$. \\
 Indeed (i) is a direct consequence of \cite[Lemma 1.1]{red} (or can be easily seen directly)  and (ii) is a consequence of (i).
\end{example}
Remark that Examples \ref{ex1}  and \ref{basic ex} apply when $_RM={_R}R$ and $R$ is a finite product of   division rings and $R$ is a two by two matrix ring over a division ring, respectively.

We close this section with the following observation:
\begin{theorem}\label{a6} $\mS (M)$ and $\mN(M)$ are connected graphs of diameter not bigger than 3.
 \end{theorem}
 \begin{proof} Suppose there are nontrivial submodules  $L \ne K$ such that $L$ and $K$ are not adjacent. Let $C, D$ denote    complements to  $L+ K$  and $L\cap K$ in $M$, respectively.

  If $L\cap K=0$, then
    we have the path $L\backsim K\oplus C \backsim L \oplus C \backsim K$.

  If $L\cap K\neq 0$, then we have the path $L\backsim D \backsim K$. Notice  all  submodules appearing  above are not essential, so the thesis follows.
 \end{proof}
  \section{On vertices of finite degree }
 The aim of this section is to characterize modules  $M$ such that all vertices of the proper sum-essential $\mN(M)$     are of finite degree and to determine submodules of degree one. Let us begin   the following theorem.

 \begin{theorem}\label{thm finite deg S(M)}
 For a module $M$ the following conditions are equivalent:
 \begin{enumerate}
   \item  $M$ has only finitely many submodules;
   \item  $\Delta(\mS (M))$ is finite;
   \item  Every vertex of $\mS (M)$ is of finite degree;
    \item Either   \emph{(i)} $M$ contains a proper essential submodule $Q$ of finite degree in $\mS(M)$ or
  \emph{(ii)} There exist simple submodules $S_1,\ldots,  S_n$ of $M$ such that $M=\bigoplus_{i=1}^nS_i$ and   $|Hom(S_i, S_j)|< \infty$,    for all $1\leq k< l\leq n$.
 \end{enumerate}
  \end{theorem}

 \begin{proof} Implications $(1)\Rightarrow (2)\Rightarrow(3)$ are  trivial.

 $(3)\Rightarrow(4)$. Suppose $M$ does not contain a proper essential submodule, i.e. $M$ is semisimple.     Let   $M=\bigoplus_{i\in I} S_i $ be a direct sum   of simple modules $S_i$, where $i$ ranges over an index set $I$. Pick $k\in I$.  Since, for any proper subset $J$ of $I\setminus\{k\}$,  $\bigoplus_{k\ne i\in I}S_i$ is adjacent to
 $S_k\oplus \bigoplus_{i\in J}S_i$, the set $I\setminus\{k\}$ has finitely many subsets, so $I$ is finite, say $I=\{1,\ldots, n\}$.

  Let us fix   $1\le k<l\le n$  and set $N=\bigoplus_{i\ne l}^nS_i$. Then, for any nonzero submodule $B$ of $S_k\oplus S_l$ different from $S_k$, we have $B+N=M$. Since  $\deg_{\mS(M)}(N)$ is finite,   Example \ref{basic ex} shows that  $|Hom(S_k, S_l)|< \infty$ for all $1\leq k<l\leq n$, i.e. (4)(ii) holds.

  $(4)\Rightarrow(1)$. If $M$ is as in (4)(i), then $M$ has finitely many submodules. Suppose (4)(ii) holds and  $M=\bigoplus_{i=1}^nS_i$.   Since $M$ is semisimple, it is enough to see that $M$ has only finitely many simple submodules.  Let $\mathcal{S}_1$ denote the set of all submodules isomorphic to $S_1$  and  $\pi_{S,i}\colon S\rightarrow S_i$, where $S\in \mathcal{S}_1$ and $1\leq i\leq n$. Set $I=\{i\mid \pi_{S,i} \ne 0\}$. If $|I|=1$, then $S=S_1$. Suppose $|I|>1$. Then $S_i\simeq S_j$, for $i,j\in I$ and any $S\in \mathcal{S}_1$  is a submodule of $V=\bigoplus_{i\in I}S_i$. By (4)(ii), $|End(S_i)|< \infty$ so also $End(V)$ is finite. This, in particular implies that $\mathcal{S}_1$ is finite. Replacing $S_1$ by any $S_k$, $1\le k\le n$, in the above consideration we see  that $M$ has finitely many simple submodules and (1) follows.
  \end{proof}
  For obtaining a similar result for the proper sum-essential graph we will need the following lemma:
\begin{lem}\label{lem finite deg P}
Suppose that  $\deg_{\mN(M)}(A)<\infty $, for any   $A\in \mV(\mN (M))$. Then $soc(M)\ess M$ and  $soc(M)$ contains only  finitely many   submodules.
\end{lem}
\begin{proof}
  Suppose $\mN (M)$ is not an empty graph. Let us observe first
   that $\ud(M)$ is finite. Indeed in case $M$ would contain nonzero
    submodules $A_i$ such that  $\bigoplus_{i=1}^\infty A_i\ess M$,
     then $N=\bigoplus_{i=2}^\infty A_i$ would be  adjacent to infinitely
     many vertices $A_1\oplus A_k$, for $k\in \N$.  This is impossible as
      $\deg N<\infty$.    Let $U$ be a uniform submodule of $M$ and $C$ a
       complement to $U$ in $M$. Then, for any nonzero submodule $A$ of $U$,
        $A\oplus C\ess M$, as $A\ess U$. This shows that every nonzero uniform
        submodule contains only finitely many submodules, in particular,  it contains
        a simple submodule. Hence $soc(M)\ess M $, as $\ud(M)<\infty$.  This means that
         every vertex of the graph $\mN(soc(M))=\mS(soc(M))$ is of finite degree, as $\mN (M)$
         has this property. Now, applying Theorem \ref{thm finite deg S(M)} to $\mS(soc(M))$ we
          get  $|\mN(soc(M))|<\infty$, i.e. the thesis holds.
\end{proof}
\begin{theorem}\label{thm finite deg P(M)}
For a module $M$ the following conditions are equivalent:
 \begin{enumerate}
   \item Every vertex of $\mN (M)$ is of finite degree;
   \item The graph $\mN(M)$ is finite.
 \end{enumerate}
\end{theorem}
\begin{proof}
 $(1)\Rightarrow(2)$ In virtue of Lemma \ref{lem finite deg P}, it is enough
 to show that for every proper  submodule $P$ of $soc(M)$ the set
  $\mathcal{A}_P=\{N<M\mid N\cap soc(M)=P\}$ is finite. Let    $C_P$ be a complement to $P$ in $M$.
  Then $C_P\ne 0$ as $P$ is a proper submodule of $soc(M)$. Moreover $C_P$
    is adjacent to any $N\in \mathcal{A}_P$ and $|\mathcal{A}_P|<\infty$ follows, as $\deg_{\mN(M)}(C_P)<\infty$.

 The implication $(2)\Rightarrow(1)$ is clear.
\end{proof}

\begin{lem} \label{lemma deg1} Let $B  \in \mV(\mS (M))$.
 If $\deg_{\mS (M)}(B)=1$, then either  $B$ is a simple module or  there exists a simple submodule $A\subset  B$ such that  $\mV(\mS (M))=\{A,B\}$.
  \end{lem}
  \begin{proof} Suppose  $B$ is not simple and let $A$ be a nonzero proper submodule of $B$.   If $B$ is essential in $M$, then clearly the second case holds.

  Thus we may assume that   $B$ is not essential in $M$. Let $C$ be a complement to
$B$.  If $A\oplus C$ is a proper submodule of $M$, then $B$ is adjacent to $C$ and $A\oplus C$. This yields $A=0$, which is impossible. Thus $A\oplus C=M$. Then also $B\oplus C=M$ and $A=B$ follows, a contradiction. This shows that if  $B$ is not essential, then it is simple.
  \end{proof}

\begin{prop}\label{semideg1}
 For a submodule $B$ of a semisimple module  $M$  the following   conditions are equivalent:
\begin{enumerate}
  \item   $\deg_{\mS (M)}(B)=1$;
  \item  $B$ is simple and has  unique nonzero complement in $M$;
  \item  $B$ is simple,   $M$ is not simple and has no other
submodules  isomorphic to $B$.
\end{enumerate}
\end{prop}
\begin{proof} $(1)\Leftrightarrow (2)$ Suppose $\deg_{\mS (M)}(B)=1$. Then, by Lemma \ref{lemma deg1}, $B$ is simple. It is also clear that the complement to  $B$ in $M$ is unique. The reverse implication is obvious.

 $(2)\Rightarrow(3)$ Let $C$ be the unique complement to $B$.
  Let $D$ be any submodule of $M$ such that $B\cap D=0$ and $E$ a
   complement to $B\oplus D$ in $M$. Since $M$ is semisimple we
    get $B\oplus D\oplus E=M$ and $D\subseteq C$ follows as $C$ is unique.
     In case $D$ would be isomorphic to $B$ then both $D$ and the submodule
     $D'=\{a+f(a)\mid a\in B\}$, where $f\colon B\rightarrow D$ is the isomorphism,
     would be contained in $C$ and $B\subseteq D+D'\subseteq C$.  This is impossible,
      so  $M$ does not contain submodules isomorphic to $B$, i.e. the statement (3) holds.

 $(3)\Rightarrow(2)$ Suppose  $B$ is a   submodule of $M$ satisfying (3).  Let
  $P,Q\subseteq M$ be complements to $B$ in $M$ and let
  $\pi\colon   M=B\oplus Q \rightarrow B$ denote the projection onto $B$ along $Q$.
  If $\pi(P)\ne 0$,  then $\pi(P)=B$ as
$B$ is simple. Then, as $P$ is semisimple, there exists a submodule
$S\subseteq P$ such that  $P=(\ker \pi\cap P)\oplus S$. Thus $S$ is
isomorphic to $B$  and $B=S\subseteq P$ follows, which is
impossible.
 Hence  $\pi(P)=0$ and consequently  $P\subseteq Q=\ker \pi$ and
maximality of $P$ gives $P=Q$, i.e. (2) holds.
\end{proof}

As a direct application of the above proposition and Lemma \ref{lemma deg1} we obtain the
following corollaries:
\begin{cor}
 Let $R$ be a semisimple ring and $B$ a minimal left ideal of $R$. Then $\deg_{\mS(R)}(B)=1$
 if and only if $B$ is a two-sided ideal of $R$. In particular, every minimal left ideal of $R$
 is of degree 1 in $\mS(R)$  if and only if $R$ is a direct product of division rings.
\end{cor}
\begin{cor}
The sum-essential graph $\mS (M)$ possesses a vertex  of degree 1 if and only if  either
  $M$ is a semisimple module  possessing   a  proper simple submodule $B$ such that $M$ does not contain other  submodules isomorphic to
  $B$ or  $M$ is a chain module with the graph with two vertices \emph{(}then any vertex is
  of degree 1\emph{)}.
 \end{cor}

Our next goal is to describe, in Theorem \ref{Theorem degree 1 gen. case},  modules $M$ such that $\mN(M)$ contains a vertex of degree one. For this we need a series of lemmas.

\begin{lem}\label{a08} Let $U$ be a   proper uniform submodule of $M$ and $S=U\cap soc(M)$. Suppose that   any submodule $D$  of $M$ satisfying $U\cap D=0$ is contained in $soc(M)$. Then:
\begin{enumerate}
  \item   $U+soc(M)\ess  M$;
  \item  If $S=0$, then $U'=soc(M)$ is the unique complement to $U$ in
 $M$;

\item  If $S\neq 0$, then $S$ is simple and  a submodule $C$ is a complement to $U$ in $M$ if and only if $C$ is the complement to $S$ in $soc(M)$. In particular, if $C$ is a complement to $U$ in $M$, then $C\oplus S=soc(M)$.
\end{enumerate}
\end{lem}

\begin{proof}
Let us notice that, due to our assumption, any complement   $U'$   of $U$ in $M$ is contained in the socle of $M$. In particular (1) holds.  The above also implies (2).

   Suppose now that $S \neq 0$.
Then $S$ is a simple module, as $U$ is uniform. Let  $U'$ be a complement to   $U$ in $M$ and $W$ a complement to $S$ in $soc(M)$.  We know that $U'$ is semisimple, so $U'$ is also a complement  of $S$ in $soc(M)$ (as $S\oplus U'\ess  U\oplus U'\ess  M$). Moreover,  if $W\subseteq U'$, then $W=U'$, i.e. $W$ is a complement to $U$ in $M$. This yields (3).
\end{proof}
In the following lemma we collect basic properties of vertices of degree one in $\mN(M)$.
\begin{lem}\label{degree 1 lemma 1}  Let $U\in \mV(\mN (M))$. If $\deg_{\mN (M)}(U)=1$, then:
\begin{enumerate}
  \item  $U$ is a uniform module;
  \item   $\deg_{\mN (M)}(B)=1$, for every vertex $B\subseteq U$;
  \item  The complement to $U$ in $M$ is a semisimple module;
  \item   Let  $E$ be a submodule of $M$ such that $E\cap U \neq  0$ and $E+U\ess M$, then $soc(M)\subseteq E$.
  \end{enumerate}
  \end{lem}

\begin{proof}
Suppose $\deg_{\mN (M)}(U)=1$ and let $C$ be the  complement to  $U$ in $M$. Then, for any  nonzero submodule $B$ of $U$,  $ U+(B+C)\ess M$. Since   $\deg_{\mN (M)}(U)=1$ and $C\ne B+C$,    $B+C\ess M$ follows.
 This easily yields (1) and (2).

(3)   If $X$ is a proper essential submodule of $C$, then $X \oplus U \ess  M$, as $U \oplus X \ess  U\oplus C\ess  M$. Therefore, as $\deg_{\mN (M)}(U)=1$,   $C$ has no proper
essential submodules so is semisimple and (3) holds.

(4) By (1), $U$ is  uniform. Let $E$ be a submodule of $M$
satisfying assumptions of (4) and $S$ be  a simple submodule of $M$.
If $S\subseteq U$, then $S=S\cap (U\cap E)\subseteq E$.

Now suppose that $S\cap U=0$. If also  $S\cap E=0$  then $E$ would
not be essential in $M$ and $U$ would be adjacent to two different
vertices, $E$ and the complement
    of $U$ which contains $S$.
   Hence $S\subseteq E$ and  $soc(M)\subseteq   E$  follows.
\end{proof}

Lemmas \ref{a08} and \ref{degree 1 lemma 1} give immediately the following
\begin{cor}\label{cor. description of complement}
   Suppose   $U\in \mV(\mN (M))$ is a vertex of degree one.  Let  $C$ be the unique complement
   to $U$ in $M$  and $S=U\cap soc(M)$.   Then:
 \begin{enumerate}
 \item If $U\cap E=0$, for a submodule $E$ of $M$, then $E\subseteq soc(M)$;
   \item $C$ is the unique complement to $S$ in $soc(M)$;
    \item  $S=0$ iff    $C= soc(M)$ \emph{(}in this case $\ud(M)=\ud(soc(M))+1$\emph{)};

   \item  $S\ne 0$  iff $C\ne soc(M)$ \emph{(}in this case  $\ud(M)=\ud(soc(M)) $\,\emph{)};
    \end{enumerate}
   \end{cor}
   \begin{proof}
    (1) is a direct consequence of Lemma \ref{degree 1 lemma 1}(3).  This statement yields (2) if $S=0$.  For $S\ne 0$, (2) is just Lemma \ref{a08}(3).  Now it is easy to complete the proof of the corollary.
   \end{proof}

 The statement (2) of Lemma \ref{degree 1 lemma 1} yields also:
\begin{cor}\label{a03} Suppose that the graph $\mN (M)$
 has only finitely many vertices of degree 1,
 then every such   vertex   contains a   simple
submodule of $M$.
\end{cor}

In the sequel  we will need yet another lemma.
  \begin{lem}\label{lemma on elements} Suppose $U$ is a nontrivial uniform submodule of a
   left $R$-module $M$ with a nonzero complement $C$, where $C\subseteq soc(M)$.
   Suppose there exist a
  simple submodule $F\nsubseteq U$, $0\neq f\in F$ and $0\neq u\in U$ such that
  for any $r\in R$, if $ru=0$ then $rf=0$. Then $\deg_{\mN (M)}(U)>1$.
  \end{lem}
  \begin{proof}  If $U$ has two different complements, then the result is clear. Thus let us assume that  $C$ is the unique complement to $U$ in $M$.
  Since $U$ is uniform,
  $Ru \oplus C \ess  U \oplus   C \ess  M$  and
  $Ru+C \ess  M$ follows.

  Let $u,f \in M$  be as in the formulation of the lemma and $r,s\in R$.
    Notice that if  $ru=su$ then $(r-s)u=0$, so also $(r-s)f=0$, i.e.
  $r(u+f)=s(u+f)$. This shows that the map
    $\varphi\colon  Ru \rightarrow R(u+f)$ defined by
   $\varphi(ru)=r(u+f)$ is a well-defined epimorphism of $R$-modules.   Moreover, since   $ru\in\ker\varphi$ gives $ru=-rf\in U\cap F=0$, $\varphi$ is an isomorphism.
   Observe that $0 \neq f \in R(u+f)+Ru$
  and hence, as $F$ is simple, we have $R (u+f) + Ru=Ru \oplus F$. Thus $R(u+f)+W+
  Ru \ess M$, where $C=W\oplus F$. Notice that $R(u+f)+W$ is not
  essential in $M$. Otherwise, $F\subseteq R(u+f)+W$ and then
  $f=ru+rf-w$, for some $r\in R$ and $w\in W$.
   So $ru=(f-rf)+w\in C \cap Ru=0$. Thus $ru=0$ and by assumption
   $rf=0$ follows. Now,  from $f=ru+rf-w$, we get $f=-w\in W$, a contradiction.
  Therefore $R(u+f)+W$ is not essential in $M$ and the essentiality of $R(u+f)+W
  +Ru$ in $M$ shows that $U$ is adjacent with two distinct
  vertices $C$ and $R(u+f)+W$. So $\deg_{\mN (M)}(U)>1$.
  \end{proof}
Now we are ready to prove the following theorem:

\begin{theorem}\label{Theorem degree 1 gen. case} Let
 $M$ a left $R$-module  and $U \in \mV(\mN (M))$.

The following   conditions are equivalent:

\begin{enumerate}
  \item  $\deg_{\mN (M)}(U)=1$;
  \item  \begin{enumerate}
   \item[(i)]$U$ is uniform  and the complement to $U$ in $M$ is unique and is semisimple;
\item[(ii)]  For every simple submodule
$F\nsubseteq U$,  $ 0\neq f\in F$ and  $ 0\neq  u\in U$ there exists an element
$r\in R$ such that $ru=0$ and $rf\neq 0$;
  \end{enumerate}
  \item  \begin{enumerate}
   \item[(i)]$U$ is uniform  and the complement to $U$ in $M$ is unique and is semisimple;

\item[(ii)]If $U\cap E \neq 0$ and
$U+E \ess  M$, for a submodule $E$ of $M$, then
$soc(M) \subseteq  E$;

  \end{enumerate}
  \item For any submodule $E$ of $M$ the following conditions hold:
\begin{enumerate}
 \item[(i)] If $U\cap E=0$  then $E\subseteq soc(M)$;
 \item[(ii)]  $U\cap E\ne 0$ and $U+E\leq_eM$, then $soc(M)\subseteq E$;
 \item[(iii)] $U$ is uniform and  $U\cap soc(M) $  has unique complement in $soc(M)$.
\end{enumerate}
\end{enumerate}
\end{theorem}

\begin{proof}
  $(1) \Rightarrow (2)$ Suppose $\deg_{\mN (M)}(U)=1$. Then part (i)
holds by Corollary  \ref{cor. description of complement} and part (ii) is a direct consequence of  Lemma \ref{lemma on elements}.

  $(2) \Rightarrow (3)$ Let $F$ be a simple submodule of $M$.  If $F\nsubseteq U$, then  the condition (2)(ii)
  implies that $R(u+f)\cap F\neq0$, for any   $0\ne u\in U$ and  $0\ne f\in F$.
Hence   $F\subseteq R(u+f)$. This shows that $R(u+f)=Ru\oplus F$,   for every simple submodule $F\nsubseteq U$, $0\neq f\in F$ and
    $ u\in U$.

    Let $E$ be a submodule of $M$ such that
$E+U\ess M$ and $U\cap E\ne 0$. If $F\subseteq U$  then, as $U$ is uniform, $F\subseteq U\cap E$, so $F\subseteq E$. Suppose $F\nsubseteq U$ and pick $0\ne f\in F$. Then $f=u+e$, for some suitable $u\in U$ and $e\in E$. If $u=0$ then $f=e\in E$ and $F\subseteq E$ follows. If $u\ne 0$, then $E\supseteq R(-u+f)=Ru\oplus F$. Thus in any case $F\subseteq E$ and (ii) follows.

  $(3) \Rightarrow (1)$ Let $C$ be a complement to $U$ in $M$. By (3)(i), $C$ is unique and   semisimple.   Let $E $ be a submodule of $M$ such
that $E+U\ess M$. If $E\subseteq C$, then $E=C$ as $C$ is semisimple. Thus, to prove the implication, it is enough to show that if $E\nsubseteq C$, then $E$ has to be essential in $M$. To this end, suppose $E\nsubseteq C$. Then, as the complement $C$ is unique,  $E \cap U \neq 0$. Therefore, by (3)(ii),  $soc(M)\subseteq E$. Notice also that the condition (3)(i) guarantees that $U$ satisfies assumptions of Lemma \ref{a08}. Let $S=U\cap soc(M)$.

 If $S=0$ (i.e. $C=soc(M)$).
Then the sum $(U\cap E)+soc(M)$ is direct.  By (i),  $U$ is uniform,  hence $(U\cap
E) \oplus  soc(M) \ess  U\oplus soc(M) \ess M$ . This
implies $E\ess M$, as $(U \cap E)\oplus soc(M) \subseteq E$.

 If $S\ne 0$ then Lemma \ref{a08}(3) implies that  $S$ simple and $C$ is the unique complement to $S$ in $soc(M)$. Then $soc(M)=S\oplus C\ess U\oplus C\ess M$ and $E\ess M$ follows, as $soc(M)\subseteq E$. This completes the proof of (1).

 We have already noticed that (3)(i) implies that $U$ satisfies the assumptions of   Lemma \ref{a08}. Then  the equivalence $(3)\Leftrightarrow (4)$ is a direct consequence of this lemma.
\end{proof}
\begin{remark}\label{rem uniqness}
   Using Proposition \ref{semideg1}, one can see  that  the condition (4)(iii) in the above theorem is equivalent to \\
(iii')  $U$ is uniform and $M$ has no other submodules isomorphic to $U\cap soc(M)$.
 \end{remark}
 The above remark together with Theorem \ref{Theorem degree 1 gen. case} and Proposition \ref{semideg1} yield:
\begin{remark}\label{rem form of complement}
 Suppose $\deg_{\mN(M)}(U)=1$.  Then $U$ is only adjacent to $C= \sum_{ F\;\mbox{\tiny  simple } F\neq S}F$, where $S=U\cap soc(M)$.
\end{remark}
Of course it may happen that the   complement to a uniform nonessential
 submodule $U$ of $M $  is unique  but $\deg_{\mN(M)}(U)>1$.
Indeed, let $R=\Z$ and $M=\Z\oplus \Z_2$. Then $R(1,0)$ has unique complement
(equal to $ R(0,1)=soc(M)$) but  $ R(1,0)$ is adjoint to any submodule $R(k,1)$, $k\in \N$. Thus $\deg_{\mN(M)}(R(1,0))=\infty$.

The following example and proposition offer  various possible interrelations between   submodules of degree one in a module $M$.
\begin{example}\label{exam5}
\begin{enumerate}\item Let $R = \mathbb{Z} \oplus \mathbb{Z}_2$, $M=_RR$ and $U\in \mN(M)$.
Then $\deg_{\mN(M)}(U)=1$  if and only if $U \subseteq R(1,0)=\Z$.
In this case  $U\cap soc(M)=0$.
  \item  Let $p,q \in \N$ be prime numbers, $R=\Z$ and $M=\Z_{p\infty}\oplus \Z_q$,
   where  $\Z_{p\infty}=\Z[\frac{1}{p}]/\Z$ is the $p$-quasicyclic group.
     Then $soc(M)=\Z_p\oplus \Z_q$. If $p\ne q$, then $\deg_{\mN(M)}(\Z_q)=\infty$ and
       a submodule $U\in \mN(M)$ is of degree 1  if and only if $U\subseteq \Z_{p\infty}$.
           If  $p=q$, then $\mN(M)$ has no vertices  of degree one.
  \item  Let $M$ be the $\mathbb{Z}$-module
 $\mathbb{Z}_4\oplus\Z_3$ . Then $soc(M)=2\Z_4\oplus \Z_3$, $\Z_4$ and $2\Z_4$ are
 the only submodules of degree one,  while $\deg_{ \mathcal{P}_ \mathbb{Z}  (M)} (\Z_3)=2$.
 \item Let $P\subseteq \N$ denote the set of primes. Then any simple submodule of
  the $\Z$-module  $\bigoplus_{p\in P}\Z_p$  is of degree 1.
\end{enumerate}
 \end{example}

\begin{prop}\label{a04} Suppose there exist two distinct submodules $A$ and $B$ of a
module $M$ such that $\deg_{\mN (M)}(A)=\deg_{\mN (M)}(B)=1$. Then:
\begin{enumerate}
  \item  If $A\cap B=0$ then $A$ and $B$ are simple non isomorphic modules;
  \item  If $A\cap B\neq 0$ then $\deg_{\mN (M)}(A+B)=1$;
  \item  If $A+B \ess  M$  then
$A, B$ are simple non isomorphic and $soc(M)=A+B$.
\end{enumerate}
\end{prop}
 \begin{proof}
(1) Suppose $A\cap B=0$. Then, by Theorem \ref{Theorem degree 1 gen. case}(4) applied to $U=A$ and $U=B$, we see that $A$  and $B$ are simple modules. Moreover, due to Remark \ref{rem uniqness}, $A\not \simeq B$.

 (3) is an immediate consequence of Remark \ref{rem form of complement}.

(2) Suppose  $A\cap B \neq 0$. If $A$ and $B$ are simple, the thesis is clear. Suppose     $B$ is not simple and let $S=A\cap soc(M)=B\cap soc(M)$.   Then,
 by (3) $A+B$ is not essential in $M$. Moreover   $A$,  $B$ and $A\cap B$  have the same unique complement $C$ described in Remark \ref{rem form of complement}.

Let $D$ be a vertex of $\mN (M)$, which is adjacent  to $A+B$.
 Applying Theorem \ref{Theorem degree 1 gen. case} to $U=A$ and $E=B+D$
  we obtain $B+soc(M)\subseteq  E=B+D$. Hence, as  $B+soc(M)\ess  M$, also
   $B+D\ess  M$ and  $\deg_{\mN(M)}(B)=1$  gives $D=C$, i.e. $\deg(A+B)=1$.
 \end{proof}
 As we have seen earlier, submodules of $M$ which are
  of degree one  as elements of $\mN(M)$ are uniform. By the above
   lemma, if $A$ and $B$ are such submodules having nonzero intersection then
    $A+B$ is a uniform submodule of $M$. Notice that, in general, a sum of
     two uniform submodules having nonzero intersection is not always uniform.

We close this section with the following theorem:
\begin{theorem}\label{Thm largest submod deg 1} Let $M$ be a module  containing a submodule $A$ of degree 1 in  $ \mN(M) $. Then either all submodules of degree 1 are simple or $M$ contains the unique largest submodule of degree 1.
\end{theorem}
\begin{proof}
 Suppose $M$ contains a not simple submodule $A$  such that  $\deg_{\mN(M)}(A)=1$. Let $B\in\mV(\mN(M))$ with $\deg_{\mN(M)}(B)=1$. Then, by Proposition \ref{a04}, we know that  $A\cap B\ne 0$ and  $\deg_{\mN(M)}(A+B)=1$. Let $S=A\cap soc(M)$ and $C$ be the   complement to $S$ in $soc(M)$. Then, by Remark \ref{rem form of complement}, $C=\sum_{ F\;\mbox{\tiny  simple } F\neq S}F$ is the unique complement to $B$ in $M$.

   Let $U$ be a uniform submodule such that $B\subseteq U$.
   We claim $C$ is a complement to $U$ in $M$. Clearly $U+C\ess  M$.
    If $U\cap C\ne 0$, then there exists a simple module $F\subseteq C$
     such that $F\subseteq U$. Then $F\subseteq B\cap soc(M)=S $, as $U$
      is uniform, a contradiction. This shows that $U\cap C=0$.
      Notice also that if  $U\oplus D \ess  M$ then $B\cap D=0$ and
      $D\subseteq C$ follows, as $\deg_{\mN(M)}(B)=1$. This proves the claim.

   Let $\mathcal{A}=\{B\in \mV(\mN(M))\mid \deg_{\mN(M)}(B)=1 \mbox{ and } A\subseteq B \}$,
           Notice that to prove the theorem   it is enough to show only that every chain in $\mathcal{A}$ is bounded, then the thesis is a consequence of Zorn's Lemma and Proposition \ref{a04}(2). To this end, let     $\{U_i\}$  be a chain in $\mathcal{A}$ and  $U= \bigcup_i U_i$. Then $U$ is uniform and the above shows that $C$ is the unique complement to $U$ in $M$. Let $F$ be a simple submodule of $M$ such that  $F\nsubseteq U$. Let   $ 0\neq f\in F$ and   $ 0\neq  u\in U$. Then $u\in U_i $, for some index $i$. Since $\deg_{\mN(M)}(U_i)=1$ we may apply Theorem \ref{Theorem degree 1 gen. case}(2)(ii) to find $r\in R$ such that     $ru=0$ and $rf\neq 0$. The above shows that $U$ satisfies  Theorem \ref{Theorem degree 1 gen. case}(2) and proves that $\deg_{\mN(M)}(U)=1$, i.e. $U\in \mathcal{A}$.   This yields the thesis.
\end{proof}
Let us notice that in case $\mN(M)$ is a finite graph, the above theorem is a direct consequence of Proposition  \ref{a04}.

\section{\bf  Complete,  $k$-regular  and triangle-free graphs   $\mS (M)$ and  $\mN (M)$}
Let us begin this section with the following easy observation:
\begin{lem}\label{unv vertex lemma}
 Let $A\in \mV(\mS(M))$ be a  universal vertex. If $A\not \ess M$, then $A$ is a simple module.
\end{lem}
\begin{proof} Notice that if $A\not \ess M$ then $A$ can not be adjacent to any of its submodules. Since $A$ is a universal vertex, the thesis follows.
\end{proof}

    The following theorem gives a characterization  of modules with a complete sum-essential graph.

 \begin{theorem}\label{thm complete graph}
 Let $M$ be a nonzero nonsimple module. Then,  $\mS (M)$ is a complete graph
 if and only if  one of the following  conditions hold:
 \begin{enumerate}
   \item  $M$ is a uniform module;
   \item  \begin{enumerate}
            \item[(i)]  Every nonzero nonessential submodule of $M$ is simple;
            \item[(ii)]  $soc(M)=S_1\oplus S_2\ess    M$, where $S_1, S_2$ are simple submodules.
          \end{enumerate}
     \end{enumerate}
 In particular, if $M$ is semisimple, then $\mS (M)$ is a
  complete graph if and only if $M$ is direct sum of two simple modules.
    \end{theorem}
    \begin{proof}
     Suppose $M$ is not uniform. Then Lemma \ref{unv vertex lemma} shows that (2)(i) holds. Now (2)(ii) is a consequence of (i) and the assumption that every vertex is universal.
    \end{proof}

 Let us notice that  every universal vertex of $\mN(M)$ is also universal
  in $\mS(M)$ and hence a nonempty graph $\mN(M)$ is  complete  if and only if
$\mS(M)$ is such.  Thus  the above theorem gives the
following corollary:
 \begin{cor}
  Suppose $M$ is not uniform.  Then,  $\mN (M)$ is a complete graph
 if and only if   the following  conditions hold:
 \begin{enumerate}
            \item[(i)]  Every nonzero nonessential submodule of $M$ is simple;
            \item[(ii)]  $soc(M)=S_1\oplus S_2\ess    M$, where $S_1, S_2$ are simple submodules.
          \end{enumerate}
 \end{cor}
 If $M$ does not contain proper essential submodules (i.e. if $M$ is semisimple) something more can be said. The last statement of the following corollary is a consequence of Example \ref{basic ex}.
\begin{cor} \label{b6} For a semisimple module $M$, the following conditions are equivalent:
   \begin{enumerate}
     \item  $\mN (M)$ is a complete graph;
     \item  $\mN(M)$ contains a universal vertex;
     \item  $M=S_1\oplus S_2$ for some simple modules
$S_1$ and $S_2$.
   \end{enumerate}
   In this case $\mS (M)$ is a complete graph and $|
\mV(\mN (M))|=| \mV(\mS (M))| =| Hom(S_1,S_2)|+1$.
\end{cor}

 \begin{cor}\label{b2}Let $R$ be a semisimple ring with a nontrivial left ideal. Then
  $\mS(R)$ is a complete graph if and only if $R\cong M_2(D)$ or $R\cong D_1\times
  D_2$, where $D$, $D_1$ and $D_2$ are division rings.
 \end{cor}
In Theorem \ref{thm complete graph}  a characterization of modules  with a   complete sum-essential graph was given.  The following theorem shows that    $k$-regular sum-essential graph  have to be complete.
  \begin{theorem}\label{b5}
 Let $M$ be a nonzero nonsimple module and $k\in \N$. The following conditions are equivalent:
 \begin{enumerate}
   \item  The graph $\mS (M)$ is $k$-regular;
     \item $\mS (M)$ is a complete graph and $|\mV(\mS(M))|=k+1$.
 \end{enumerate}
   \end{theorem}

 \begin{proof}
$(1)\Rightarrow(2)$ Let $\mS (M)$ be a $k$-regular. Suppose $M$ contains a proper essential submodule $E$. Then $E$ is adjacent to any $ A\in \mV(\mS(M))\setminus \{E\}$. Then $|\mV(\mS(M))|=k+1$ and $\mV(M)$ is a complete graph.

 Assume now that $M$ does not have proper essential submodule, i.e. $M$ is semisimple. Let $S$ be a simple submodule of $M$. Then $k$ is the number of maximal submodules of $M$ not containing $S$. Let $N_1,\ldots ,N_k$ be  such maximal submodules.  Then $N_1$ is adjacent to  any $N_i$, $1<i\leq k$ and to arbitrary proper submodule of $M$ containing $S$. The fact that  $\mS(M)$ is $k$-regular implies that $M$ has no proper submodules containing $S$. Thus the semisimple module $M=S\oplus Q$ for suitable simple submodule $Q$ of $M$. Now the result is a consequence of Example \ref{basic ex}.

  The implication $(2)\Rightarrow(1)$ is clear.
 \end{proof}
The following theorem presents the structure of modules having triangle-free proper sum-essential graphs.

 \begin{theorem}\label{thm S triangle free}
 Let $M$ be a module such that $\mid  \mV(\mS (M)) \mid\geq 2$. The following conditions are equivalent:
 \begin{enumerate}
   \item The graph $\mS (M)$ is    triangle-free;
   \item  Either $M$ is a direct sum of two non-isomorphic simple modules or $M$ is a chain module with exactly two proper submodules.
   \item   $\mS (M)\simeq \mathcal{K}_2$ the complete
 graph with two vertices.
 \end{enumerate}
   \end{theorem}
\begin{proof}
$(1)\Rightarrow(2)$ Suppose $\mS (M)$ is a  triangle-free.  If $M$ would contain   an essential  direct sum $P\oplus Q\oplus S $ of nonzero submodules,
then  $P\oplus Q,P\oplus S, Q\oplus S $ would  form a triangle,  a contradiction. Hence
$\ud(M)\leq 2$.

If $\ud(M)=1$,  then any two vertices are adjacent, so $M$ has to be a chain module with two vertices only.

Suppose  $\ud(M)=2$ and let $P$  be a uniform submodule  and $Q$ its complement in  $M$.
  If $P$ would contain a proper nonzero submodule $S$ then $S\oplus Q,P$ and $Q$ would form a
  triangle. Thus $P$ and $Q$ are simple and $P\oplus Q\ess M$. If $P\oplus Q$ would be a proper
   submodule, then we would have a triangle with vertices $P, Q, P\oplus Q$. Therefore
    $M=P\oplus Q$ in this case. Moreover    $P$ and $Q$ are not isomorphic, as three distinct simple submodules
   would form a triangle.  This completes the proof.

Implications $(2)\Rightarrow(3)\Rightarrow(1)$ are clear.
\end{proof}

\begin{cor}\label{cor girth of S}
Let $M$ be a nonzero nonsimple module. Then the girth of the graph $\mS (M)$ is either $3$ or $\infty$.

\end{cor}
\begin{proof}
  If $\mS (M)$ contains a triangle then
$g(\mS (M))=3$. If $\mS (M)$ is triangle-free then,
   by  Theorem \ref{thm S triangle free}, $\mS (M) = \mathcal{K}_2$ or $\mS (M)$ is a single point
    and    $g(\mS (M))=\infty$.
\end{proof}
Remark that bipartite graphs are triangle-free, so  the above corollary applies in this case.

\begin{definition}\label{S.d.}
We say that submodules  $A$ and $B$   of a module $M$  are \it{strongly disjoint}   if  there are no nonzero isomorphic submodules $X, Y $ of $M$ such that $X\subseteq A$ and $Y\subseteq B$.
\end{definition}

The following lemma offers some other characterizations of strongly disjoint submodules.
\begin{lem}\label{sd2}
For submodules $A$ and $B$ of    a left $R$-module $M$   the
following conditions are equivalent:
\begin{enumerate}
  \item  $A$ and $B$ are strongly disjoint;
  \item For any   nonzero elements $a\in A$ and $b\in
B$, $\ann(a)\neq \ann(b)$;
  \item \begin{enumerate}
          \item[(i)]  $A\cap B=0$;
          \item[(ii)] If $U$ is a nonzero submodule
of $A+B$, then $U\cap A\neq0$ or $U\cap B\neq0$.
        \end{enumerate}
\end{enumerate}
\end{lem}
\begin{proof}
  The implication  $(1)\Rightarrow (2)$ is clear as  any cyclic  $R$-module $Rc$  isomorphic to $ R/\ann(c) $.

  $(2)\Rightarrow (3)$ Suppose (2) holds. Then clearly $A\cap B=0$. Let $U$ be a nonzero submodule of $A+B$ and $ 0\ne u=a+b$, with $a\in A$ and $b\in B$.  If one of $a$  and $b$ is equal to zero, then the result holds. If both elements are nonzero we can use (2) to find an element $r\in R$ which belongs to exactly one of the sets $\ann(a)$ and $\ann(b)$. Then  either $  ru=rb\ne 0  $ or $ru=ra\ne 0$ and the result follows.

$(3)\Rightarrow (1)$ Let $X\subseteq A$ and $Y\subseteq B$ be isomorphic submodules. Let us consider  the submodule $U=\{x+\phi(x)\mid  x\in X \}\subseteq A+B$, where $\phi\colon X\rightarrow Y$ is a given isomorphism.  Let  $  u\in U$ and   $x\in X$ be such that $u=x+\phi(x)$. If $u\in U\cap A$, then  $\phi(x)=u-x\in A\cap B=0$ and $U\cap A=0$ follows. Similarly, if $u\in U\cap B$, then $x=u-\phi(x)\in A\cap B=0$ yields   $U\cap B= 0$.    Therefore, by (3), $X \simeq U=0$, i.e. $A$ and $B$ are strongly disjoint.
\end{proof}

\begin{theorem}\label{triangle1} Let $M$ be a module with nonempty graph $\mN (M)$.
The following   conditions are equivalent:
\begin{enumerate}
  \item  $\mN (M)$ is triangle-free;
  \item \begin{enumerate}
          \item[(i)]   $\ud(M)=2$ \emph{(}so any nonzero, nonessential submodule is uniform\emph{)};
          \item[(ii)]  If $A,B \in \mV(\mN (M))$ and $A+B\ess M$, then
$A$ and $B$ are strongly disjoint.
        \end{enumerate}
        \item If $A,B \in \mV(\mN (M))$ and $A+B\ess M$, then
$A$ and $B$ are strongly disjoint.
\end{enumerate}
 \end{theorem}

\begin{proof} (1)$\Rightarrow$(2) Assume $\mN (M)$ is triangle-free. By assumption  $M$
 is not uniform. The argument used at the beginning of the proof of Theorem \ref{thm S triangle free} shows that $\ud(M)\leq 2$, i.e. $(2)(i)$ holds.

 Let  $A,B \in \mV(\mN (M))$ and $A+B\ess M$.    Let $N$=$A\cap
B $ and $N'$ be the complement to $N$ in $M$.  If $N$ would be nonzero, then $N'$ would be a proper nonessential submodule of $M$ and we would have a triangle with vertices  $A,B, N' $ (as $A+B\ess M$).   This is impossible, so $A \cap B=0$. Let $U$ be a uniform nonzero submodule of $A+B$. Using (2)(i) we see that $A\cap U\ne 0$ or $B\cap U\ne 0$, as otherwise   $A,B,U$ would form a triangle. Now (2)(ii) is a consequence of Lemma \ref{sd2}.

The implication $(2)\Rightarrow(3)$ is clear.

$(3)\Rightarrow(1)$ Let $A,B,C\in \mV(\mN(M))$ be such that $A+C\ess M$ and $B+C\ess M$. In particular,  by  assumption, we have  $A\cap C=0=B\cap C$. Let $U=C\cap (A+B)$. Notice that  $A+B\not\ess M$, as otherwise $U\ne 0$ and (3) would imply that  either $U\cap A\ne 0$ or $U\cap B\ne 0$, which is impossible. This shows that $\mN(M) $ is triangle free.
\end{proof}

\begin{theorem} \label{treeta}Let $M$ be a module with $\mV(\mN (M))\ne \emptyset$. Then
the following conditions are equivalent:
\begin{enumerate}
  \item  $\mN (M)$ is a tree;
 \item If $A,B \in \mV(\mN (M))$ and $A+B\ess M$, then
$A$ and $B$ are strongly disjoint and one of $A$ and $B$ is a simple module;

   \item  $\mN (M)$ is a star graph with the
center $S$, for a simple submodule $S$ of $M$.

\end{enumerate}
 \end{theorem}
\begin{proof}
   \indent(1)$\Rightarrow$(2) Every tree is triangle-free, so the first part of the statement (2) is given by Theorem \ref{triangle1}.
   Let $A,B\in \mV(\mN(M))$ be such that  $A+B\ess M$. In case we would have proper nonzero submodules $X$ and $Y$ of $A$ and $B$, respectively,
 we would have a cycle $X \backsim B \backsim A \backsim Y \backsim X$ in a tree. Hence one of $A$, $B$ has to be a simple module, i.e. (2) holds.

 $(2)\Rightarrow(3)$ Suppose (2) holds. Then $M$ contains a simple submodule, call it $S$.  Theorem \ref{triangle1} shows that $\mN(M)$ is triangle free and  $\ud(M)=2$. In particular $M$ has at most two different simple submodules.

   If $soc(M)=S$ then, by (2),  $\mN (M)$ is a star graph with center $S$.

  Suppose  $soc(M)$ is not simple. Since $\ud(M)=2$,  $S\oplus Q=soc(M)\ess M$ where $S,Q$ are the only simple submodules of $M$.    In particular, any $A\in \mV(\mN(M))$ contains one of $S$ or $Q$.  The statement (2) implies that one of the simples, say $S$, does not have proper essential extensions in $M$.
  Therefore, if  $S\ne A\in \mV(\mN(M))$, then $Q\subseteq A$ and $\mN(M)$ is a star graph with center $S$.

 The implication $(3)\Rightarrow(1)$  is clear.
 \end{proof}
 Let us remark that the proof of the implication $(2)\Rightarrow(3)$ gives
 the module structure of modules with  $\mN(M)$ being a tree.  The $\Z$-modules
 $\Z\oplus \Z_q$ or $\Z_{p^n}\oplus \Z_q$ are examples of such modules,
  where $p,q,n\in \N$ and $p\neq q$ are prime.
We have seen in Corollary \ref{cor girth of S} that the girth of
$\mS(M)$ is equal to 3 when finite. In case of the graph $\mN(M)$ we
have the following:
\begin{cor}
Suppose  $\mN (M) $  is a nonempty graph.    Then $g(\mN (M)) \in \{3,4,\infty \}$.
\end{cor}
\begin{proof}
Assume $\mN (M)$ is a triangle-free graph, i.e. $g(\mN (M))\ne 3$.  If $\mN (M)$ is not a tree  then, using Theorems \ref{treeta} and \ref{triangle1} we can pick two strongly disjoint uniform submodules $A$ and $B$ in $\mN (M)$ such that   $A+B\ess M$ and neither $A$ nor
$B$ is simple.   Let $X$,  $Y$ be   proper nonzero submodules
of $A$ and $B$, respectively. Then
$X,Y, A,B $ form a  cycle, so $g(\mN (M))=4$.   This yields the thesis.
\end{proof}
 The girth of $\mN (M)$ can take any value as described above. If $\ud(M)>2$, then $g(\mN(M))=3$. If $\ud(M)=2$ consider the group
$M=\mathbb{Z}_{p^m}\oplus \mathbb{Z}_{q^n}$ as a
$\mathbb{Z}$-module, where $p,q\in \N$ are different primes and $n,m\geq 1$. Then any nonessential submodule of $M$ is contained either in $\Z_{p^m}$ or $\Z_{q^n}$ and  Theorem \ref{triangle1} (or direct argument) shows that $\mathcal{P}_\mathbb{Z}(M)$ is triangle free. If $n=1$, then $g(\mathcal{P}_\mathbb{Z}(M))=\infty$. If $n,m\geq 2$, then $g(\mathcal{P}_\mathbb{Z}(M))=4$.\\

  We close the paper with the following  result:
\begin{prop}\label{b7}
Suppose that the module $M$ has exactly $2\leq n<\infty$ maximal submodules. The following conditions are equivalent:
\begin{enumerate}
  \item   $\mS(M)$ is $n$-partite graph;
  \item  $M$ is a semisimple module.
\end{enumerate}
\end{prop}
\begin{proof}Let $M_1,\ldots,M_n$ be the maximal submodules  of $M$, $J=J(M)=\bigcap_{i=1}^n  M_i$ and $N$ be the complement to $J$ in $M$.

 $(1)\Rightarrow(2)$    In case $J$ would be nonzero then $N, M_1,\ldots ,M_n$ would generate a complete subgraph of $\mS(M)$. Since $\mS(M)$ is $n$-partite graph, this is impossible and $J=0$ follows. This yields the thesis.

 $(2)\Rightarrow(1)$ Suppose $M$ is a semisimple module. For
  $1\leq k\leq n$, let us define $V_k=\{A\in \mV(\mS(M))\mid A\subseteq M_k$
  and $A \nsubseteq M_i$ for $ i<k \}$. Then, using semisimplicity of $M$ it is
  easy to see that $V_1, \ldots, V_n$ form $n$-partitioning subsets of $\mV(\mS(M))$.
 \end{proof}

 \bibliographystyle{elsarticle-num}

 \end{document}